\newcommand{\bbC}{\mathbb{C}}

\newcommand{\bbP}{\mathbb{P}}
\newcommand{\bbQ}{\mathbb{Q}}

\newcommand{\bbZ}{\mathbb{Z}}

\newcommand{\Tr}{\textup{Tr}}
\newcommand{\rank}{\textup{rank}}
\newcommand{\Pic}{\textup{Pic}}

\newcommand{\DR}{\text{DR}}

\newcommand{\crys}{\text{crys}}

\newcommand{\NS}{\text{NS}}
\newcommand{\Br}{\text{Br}}

\documentclass[11pt]{amsart}
\usepackage{amscd, amsmath, amsthm, amssymb}
\newtheorem{theorem}{Theorem}[section]
\newtheorem{lemma}[theorem]{Lemma}
\newtheorem{proposition}[theorem]{Proposition}

\newtheorem{corollary}[theorem]{Corollary}
\theoremstyle{definition}     

\newtheorem{remark}[theorem]{Remark}
\numberwithin{equation}{section}

\begin{document}

\title[K3 surfaces with an order $60$ automorphism]{K3 surfaces  with an order $60$ automorphism and a characterization of supersingular K3 surfaces with Artin invariant 1}

\author[J. Keum]{JongHae Keum}
\address{School of Mathematics, Korea Institute for Advanced Study, Seoul 130-722, Korea } \email{jhkeum@kias.re.kr}
\thanks{Research supported by National Research Foundation of Korea (NRF grant).}


\date{May 2012, September 2013}
\begin{abstract} In characteristic $p=0$ or $p>5$, we show that a K3 surface with an order $60$ automorphism is unique up to isomorphism.
As a consequence, we characterize the supersingular K3 surface
with Artin invariant 1 in characteristic $p\equiv 11$ (mod 12) by
a cyclic symmetry of order 60.
\end{abstract}

\maketitle

Let $X$ be a K3 surface over an algebraically closed field $k$ of
characteristic $p\ge 0$. An automorphism $g$ of $X$ is called
\emph{symplectic} if it preserves a regular 2-form $\omega_X$, and
 \emph{purely non-symplectic} if no power of $g$ is symplectic except the identity.

Over $k=\bbC$, Xiao \cite{Xiao} and Machida and Oguiso \cite{MO}
proved that a positive integer $N$ is the order of a purely
non-symplectic automorphism of a complex K3 surface if and only if
$\phi(N)\le 20$ and $N\neq 60$, where $\phi$ is the Euler
function. On the other hand, there is a K3 surface with an
automorphism of order 60 (\cite{K} Example 3.2):
\begin{equation}\label{formula}
X_{60}:= (y^2+x^3+t_0t_1^{11}-t_0^{11}t_1 = 0)\subset
\mathbb{P}(4,6,1,1),
 \end{equation}
\begin{equation}\label{form2}
 g_{60}(t_0,t_1,x,y)=(t_0,\zeta_{60}^6t_1,\zeta_{60}^2x,\zeta_{60}^3y)
 \end{equation}
 where $\zeta_{60}\in k$ is a primitive 60th  root of unity.
The K3 surface $X_{60}$ is defined over the integers and both the
surface and the automorphism have a good reduction mod $p$ unless
$p=2$, 3, 5. 

For an automorphism $g$ of finite order of a K3 surface $X$, we
write
$${\rm ord}(g)=m.n$$ if $g$ is of order $mn$ and the
natural homomorphism $$\langle g\rangle\to {\rm GL}(H^0(X,
\Omega^2_X))$$ has kernel of order $m$ and image of order $n$. We call $n$ the \textit{non-symplectic order} of $g$.

The main result of the paper is the following.

\begin{theorem} \label{main} Let $k$ be an algebraically closed
field of characteristic $p=0$ or $p>5$. Let $X$ be a K3 surface
defined over $k$ with an automorphism $g$ of order $60$. Then
\begin{enumerate}
\item ${\rm ord}(g)=5.12$;
\item the pair $(X, \langle g\rangle)$ is isomorphic to the pair $(X_{60}, \langle g_{60}\rangle)$, i.e. there is an isomorphism $f:X\to X_{60}$ such that $f \langle g\rangle f^{-1}=\langle g_{60}\rangle$.
\end{enumerate}
\end{theorem}

 The  non-existence  of a complex K3 surface with a
purely non-symplectic automorphism of order 60 was proved by
Machida and Oguiso \cite{MO}. Their proof does not extend to the
positive characteristic case, as it uses the holomorphic Lefschetz
formula and the notion of transcendental lattice, both not
available in positive characteristic.

Theorem \ref{main} and Main Theorem of \cite{K} determine completely the list of all non-symplectic orders in characteristic $p>0$: 

\begin{corollary} \label{cor} In any fixed characteristic $p>0$, a positive integer $N$ is the
non-symplectic order of an automorphism of a K3 surface if and only if $p\nmid N$, $N\neq 60$ and $\phi(N)\le 20$.
\end{corollary}

It is well known that the Fermat quartic surface
$$x_0^4+x_1^4+x_2^4+x_3^4 = 0$$ is a supersingular K3 surface with Artin invariant 1, if the characteristic $p\equiv 3$ (mod
4). This can be seen by using the algorithm  for determining the
Artin invariant of a weighted Delsarte surface whose minimal
resolution is a K3 surface (\cite{Shioda}, \cite{Goto}). The same
algorithm shows that in characteristic $p\equiv 11$ (mod 12) the
surface $X_{60}$ is a supersingular K3 surface with Artin
invariant 1, hence is isomorphic to the Fermat quartic surface,
since a supersingular K3 surface with Artin invariant 1 is unique
up to isomorphism (\cite{Ogus}, \cite{Ogus2}).

\begin{corollary}  In characteristic $p\equiv 11$ $({\rm mod}\,\, 12)$, the Fermat quartic surface is the only K3 surface with an order $60$ automorphism.
\end{corollary}

Over $k=\bbC$,  Oguiso \cite{Og2} proved that the Fermat quartic
surface is the only K3 surface with a faithful action of a
nilpotent group of order $512=2^9$. Over $k=\bbC$, the surface
$X_{60}$ is not isomorphic to the Fermat quartic surface, as the
former admits a purely non-symplectic automorphism of order $12$,
while the latter has Picard number 20, hence by Nikulin \cite{Nik}
does not admit a purely non-symplectic automorphism of order $n$
with $\phi(n)>2$.

\begin{remark}  In characteristic $p=11$ the Fermat quartic surface
also admits a cyclic action of order 66 (Example 7.5 \cite{K}) and
a symplectic action of the simple groups $M_{22}$, $M_{11}$ and
$L_2(11)$, where $M_{r}$ is one of the Mathieu groups \cite{DK3}.
\end{remark}

Throughout this paper, whenever we work with $l$-adic cohomology we assume $l$ is any prime different from the characteristic.

\bigskip

{\bf Notation}

\medskip
\begin{itemize}
\item ${\rm NS}(X)$ : the N\'eron-Severi group of a variety $X$;
\medskip
\item $X^g={\rm Fix}(g)$ : the fixed locus of an automorphism $g$ of $X$;
\medskip
\item $e(g):=e({\rm Fix}(g))$, the Euler characteristic of ${\rm Fix}(g)$ for $g$ tame;
\medskip
\item $\Tr(g^*|H^*(X)):=\sum_{j=0}^{2\dim X} (-1)^j\Tr (g^*|H^j_{\rm et}(X,{\bbQ}_l))$.

\medskip\noindent
For an automorphism $g$ of a K3 surface $X$,
\end{itemize}
\begin{itemize}
\item ${\rm ord}(g)=m.n$ : $g$ is of order $mn$ and the representation of the group $\langle g^*\rangle$ on $H^0(X, \Omega^2_X)$ has kernel of order $m$;
\medskip
\item $[g^*]=[\lambda_1, \ldots, \lambda_{22}]$ : the list of the eigenvalues of $g^*|H^2_{\rm et}(X,{\bbQ}_l)$.
\medskip
\item $\zeta_a$ : a primitive $a$-th  root of unity in $\overline{\bbQ_l}$;
\medskip
\item $[\zeta_a:\phi(a)]\subset [g^*]$ : all primitive $a$-th  roots of unity appear in $[g^*]$ where $\phi(a)$ indicates the number of
them.
\medskip
\item $[\lambda.r]\subset [g^*]$ : $\lambda$ repeats $r$ times in
$[g^*]$.
\medskip
\item $[(\zeta_a:\phi(a)).r]\subset [g^*]$ : the list $\zeta_a:\phi(a)$ repeats $r$ times in
$[g^*]$.
\end{itemize}

\section{Preliminaries}

 We first recall the following basic
result.

\begin{proposition}\label{integral}$($3.7.3 \cite{Illusie}$)$ Let $g$ be an automorphism of a projective variety $X$ over an algebraically closed field $k$ of
characteristic $p> 0$. Let $l$ be a prime $\neq p$. Then the following hold true.
\begin{enumerate}
\item  The characteristic polynomial of
$g^*|H_{\rm et}^j(X,\bbQ_l)$ has integer coefficients for each
$j$. In particular, if for some positive integer $m$ a primitive
$m$-th root of unity appears with multiplicity $r$ as an
eigenvalue of $g^*|H_{\rm et}^j(X,\bbQ_l)$, then so does each of
its conjugates.
\item The characteristic polynomial of $g^*$ does not depend on the choice of  cohomology, $l$-adic or crystalline.
\end{enumerate}
\end{proposition}

\begin{proposition}\label{integral'}  Let $g$ be an automorphism of a projective variety $X$ over an algebraically closed field $k$ of
characteristic $p> 0$. Let $l$ be a prime $\neq p$. Then
the following hold true.
\begin{enumerate}
\item  If $g$ is of finite order, then $g$ has an
invariant ample divisor, and $1$ is an eigenvalue of
$g^*|H_{\rm et}^2(X,\bbQ_l)$.
\item  If $X$ is a K3 surface, $g$ is tame and $g^*|H^0(X,\Omega_X^2)$ has
$\zeta_n\in k$ as an eigenvalue, then $g^*|H_{\rm et}^2(X,\bbQ_l)$
has $\zeta_n\in \overline{\bbQ_l}$ as an eigenvalue.
\end{enumerate}
\end{proposition}

\begin{proof}
(1) For any ample divisor $D$ the sum $\sum g^{i}(D)$ is $g$-invariant. A $g^*$-invariant ample line bundle gives a $g^*$-invariant vector in the 2nd crystalline cohomology $H_{\rm crys}^2(X/W)$ under the Chern class map 
$$c_1: \Pic(X)\to H_{\rm crys}^2(X/W).$$ It follows that $1$ is an eigenvalue of
$g^*|H_{\rm crys}^2(X/W)$. Here $W=W(k)$ is the ring of Witt vectors.  Now apply Proposition \ref{integral}(2).

(2)  
The quotient module
$$H_{\crys}^2(X/W)/pH_{\crys}^2(X/W)$$ is a finite dimensional
$k$-vector space isomorphic to the algebraic de Rham cohomology
$H_{\DR}^2(X)$. See \cite{Illusie} for the  crystalline cohomology. It is known that the Hodge to de Rham spectral sequence  $$E_1^{t,s}:=H^s(X, \Omega_X^t)\Rightarrow H_{\DR}^*(X)$$ degenerates at $E_1$, giving
 the Hodge filtration on $H_{\DR}^2(X)$ and the following canonical exact sequences:
$$0\to F^1\to F^0=H_{\DR}^2(X)\to H^2(X, \mathcal{O}_X)\to 0$$
 $$0\to F^2=H^0(X,\Omega_X^2)\to F^1\to  H^1(X,\Omega_X^1)\to 0.$$
 In particular $g^*|H_{\DR}^2(X)$ has
$\zeta_n\in k$ as an eigenvalue. The corresponding eigenvalue of  $g^*|H_{\crys}^2(X/W)$ must be an $np^r$-th root of unity for some $r$, since $n$ is not divisible by $p$. Then $g^{p^r*}|H_{\crys}^2(X/W)$ has an  $n$-th root of unity as an eigenvalue.
Since $g$ is tame, so does $g^{*}|H_{\crys}^2(X/W)$.
\end{proof}

Recall that for a nonsingular projective variety $Z$ in
characteristic $p>0$, there is an exact sequence of
$\bbQ_l$-vector spaces
\begin{equation}\label{trans}
0\to \NS(Z)\otimes \bbQ_l \to  H_{\rm et}^2(Z,\bbQ_l) \to
T_l^2(Z)\to 0
\end{equation}
where $T_l^2(Z) = T_l(\textup{Br}(Z))$ in the standard notation in
the theory of \'etale cohomology (see \cite{Shioda2}). The Brauer
group $\Br(Z)$ is known to be a birational invariant.

\begin{proposition}\label{diminv}  Let $Z$ be a nonsingular projective variety in
characteristic $p>0$. Let $g$ be an automorphism of $Z$ of finite
order. Assume $l\ne p$. Then the following assertions are true.
\begin{enumerate}
\item  Both traces of $g^*$ on ${\rm NS}(Z)$ and on $T_l^2(Z)$ are integers.
\item  $\rank~{\rm NS}(Z)^g=\rank~{\rm NS}(Z/\langle  g \rangle).$
\item  $\dim H^2_{\rm et}(Z,{\bbQ}_l)^g=\rank~{\rm NS}(Z)^g+\dim T_l^2(Z)^g$.
\item  If the minimal resolution $Y$ of $Z/\langle  g \rangle$ has $T_l^2(Y)=0$, then $$\dim H^2_{\rm et}(Z,{\bbQ}_l)^g=\rank~{\rm NS}(Z)^g.$$
\end{enumerate}
The condition of $(4)$ is satisfied if $Z/\langle g\rangle$ is
rational or is birational to an Enriques surface.
\end{proposition}

The following is well known, see for example Deligne-Lusztig
(Theorem 3.2 \cite{DL}).

\begin{proposition}\label{trace}$($Lefschetz fixed point formula$)$   Let $X$ be a smooth projective variety over  an algebraically closed field $k$ of
characteristic $p> 0$ and let $g$ be a tame automorphism of $X$. Then $X^g={\rm Fix}(g)$ is smooth and
 $$e(g):=e(X^g)=\Tr(g^*|
H^*(X)).$$
\end{proposition}

A tame symplectic automorphism $h$ of  a K3 surface has finitely
many fixed points, the number of fixed points $f(h)$ depends only
on the order of $h$ and the list of possible pairs $({\rm ord}(h),
f(h))$ is the same as in the complex case (Theorem 3.3 and
Proposition 4.1 \cite{DK2}): $$({\rm ord}(h),
f(h))=(2,8),\,\,(3,6),\,\,(4,4),\,\,(5,4),\,\,(6,2),\,\,(7,3),\,\,(8,2).$$
Thus by the Lefschetz fixed point formula (Proposition
\ref{trace}), we obtain the following.

\begin{lemma}\label{Lefschetz} Let $h$ be a tame symplectic automorphism of a K3 surface $X$. Then
$h^*|H_{\rm et}^2(X,\bbQ_l)$ has eigenvalues
$$\begin{array}{lll} {\rm ord}(h)=2&:&[h^*]=[1,\, 1.13,\, -1.8]\\
{\rm ord}(h)=3&:&[h^*]=[1,\, 1.9,\, (\zeta_3:2).6]\\
{\rm ord}(h)=4&:&[h^*]=[1,\, 1.7,\, (\zeta_4:2).4,\, -1.6]\\
{\rm ord}(h)=5&:&[h^*]=[1,\, 1.5,\, (\zeta_5:4).4]\\
{\rm ord}(h)=6&:&[h^*]=[1,\, 1.5,\, (\zeta_3:2).4,\, (\zeta_6:2).2,\, -1.4]\\
{\rm ord}(h)=7&:&[h^*]=[1,\, 1.3,\, (\zeta_7:6).3]\\
{\rm ord}(h)=8&:&[h^*]=[1,\,
1.3,\,(\zeta_8:4).2,\,(\zeta_4:2).3,\, -1.4]\end{array}$$ where
the first eigenvalue corresponds to an invariant ample divisor.
\end{lemma}

We need the following information on a special
involution of a K3 surface.

\begin{lemma}\label{nsym2}  
Let $X$ be a K3 surface in characteristic $p\neq 2$.
Assume that $h$ is an automorphism of order $2$ with $\dim
H^2_{\rm et}(X,{\bbQ}_l)^h =2$. Then $h$ is non-symplectic and has
an $h$-invariant elliptic fibration $\psi:X\to {\bf P}^1$,  $$X/\langle h\rangle\cong {\bf F}_e$$ a rational ruled surface, and $X^h$ is either a curve of genus
$9$ which is a $4$-section of $\psi$ or the union of a section and a curve of genus $10$ which
is a $3$-section.
In the first case $e=0, 1$ or $2$, and in the second $e=4$. Each singular fibre of $\psi$
is of type $I_1$ $($nodal$)$, $I_2$, $II$ $($cuspidal$)$ or $III$, and is intersected by $X^h$ at the node and two smooth points if  of type $I_1$, at the two singular points if of type $I_2$, at the cusp with multiplicity $3$ and a smooth point if of type $II$, at the singular point tangentially to both components if of type $III$. If $X^h$ contains a section, then each singular fibre
is of type $I_1$ or $II$.
\end{lemma}

\begin{proof} Since $\dim H^2_{\rm et}(X,{\bbQ}_l)^h=2$, the eigenvalues of
$h^*|H^2_{\rm et}(X,{\bbQ}_l)$ must be $$[h^*]=[1.2,\,-1.20], \,\,{\rm so}\,\,\,\Tr
(h^*|H^*(X))=-16.$$ By Lemma \ref{Lefschetz},
$h$ is non-symplectic, thus $X^h$  is a disjoint union of smooth
curves and the quotient $X/\langle h\rangle$ is a nonsingular
rational surface.  By Proposition
\ref{diminv}, $X/\langle h\rangle$ has Picard number 2, hence is
isomorphic to a rational ruled surface ${\bf F}_e$.  Note that $e(X^h)=-16$, so $X^h$ is non-empty
and has at most 2 components. Thus $X^h$ is either a curve $C_9$ of
genus 9 or the union of two curves $C_0$ and $C_{10}$ of genus 0 and 10, respectively. 
 In the first case, the image $C_9'\subset {\bf F}_e$ of $C_9$ satisfies $C_9'^2=32$ and $C_9'K=-16$, hence $C_9'\equiv 4S_0+(4+2e)F$, where $S_0$ is the
section with $S_0^2=-e$, and $F$ a fibre of ${\bf F}_e$. Since $S_0C_9'\ge 0$, we have $e\le 2$.
In the second case, the image $C_0'$ of $C_0$ has $C_0'^2=-4$, hence $C_0'=S_0$ and $e=4$, then it is easy to see that  $C_{10}'\equiv 3(S_0+4F)$.

In characteristic $p\neq 3$ the pull-back of
the ruling on ${\bf F}_e$ gives an $h$-invariant elliptic fibration
$\psi:X\to {\bf P}^1$. Each singular fibre has at most 2 components since it is the pull-back of a fibre of ${\bf F}_e$.

In characteristic $p=3$ we have to show that the pull-back is not a quasi-elliptic fibration. Suppose it is. The closure of the cusps of irreducible fibres is a smooth rational curve and must be fixed pointwise by $h$, then the genus 10 curve must be a section of the quasi-elliptic fibration, impossible. 
\end{proof}

The following easy lemmas also will be used frequently.

\begin{lemma}\label{fix} Let $S$ be a set and ${\rm Aut}(S)$ be the group of bijections of $S$. For any $g\in {\rm Aut}(S)$ and positive integers $a$ and $b$,
\begin{enumerate}
\item ${\rm Fix}(g)\subset {\rm Fix}(g^a)$;
\item ${\rm Fix}(g^a)\cap {\rm Fix}(g^b)={\rm Fix}(g^d)$ where $d=\gcd (a, b)$;
\item ${\rm Fix}(g)= {\rm Fix}(g^a)$ if ${\rm ord}(g)$ is finite and  prime to $a$.
\end{enumerate}
\end{lemma}

\begin{lemma}\label{sum} Let $R(n)$ be the sum of all primitive $n$-th root of unity in $\overline{\bbQ}$ or in
$\overline{\bbQ_l}$. Then
$$R(n)=\left\{\begin{array}{ccl} 0&{\rm if}& n\,{\rm has\,\, a\,\, square\,\, factor},\\
(-1)^t&{\rm if}& n\,{\rm is\,\, a\,\, product\,\, of}\,\,t\,\,{\rm distinct\,\, primes}.\\
\end{array} \right.$$
\end{lemma}

The following lemma will play a key role in our proof.

\begin{lemma}\label{60} Let $g$ be an automorphism of order $60$ of a K3
surface in characteristic $p\neq 2$, $3$, $5$. If $$[g^{*}]=[1,\,\zeta_{60}:16,\,\zeta_{12}:4,\,\pm 1],$$ then
\begin{enumerate}
\item there is a $g^{}$-invariant elliptic
fibration $\psi:X\to {\bf P}^1$ with $12$ cuspidal fibres, say $F_{\infty}$, $F_0$, $F_{t_1}, \ldots, F_{t_{10}}$;
\item  ${\rm Fix}(g^{30})$
consists of a section $R$ of $\psi$ and a curve $C_{10}$ of genus
$10$ which is a $3$-section passing through each cusp with multiplicity $3$;
\item the action of $g$ on the base ${\bf P}^1$ is of order $10$, fixing $2$ points, say  $\infty$ and $0$, and makes the $10$ points $t_1, \ldots, t_{10}$ to form a single orbit;
\item ${\rm Fix}(g^{10})=R\cup\,\{{\rm the\,
cusps\, of\, the\, 12\, cuspidal\, fibres}\};$
\item ${\rm Fix}(g^{12})={\rm Fix}(g)$ and it  consists of the  
$4$ points,  $$R\cap F_{\infty},\,\,R\cap F_{0},\,\,C_{10}\cap F_{\infty},\,\,C_{10}\cap F_{0};$$
\item $[g^{*}]=[1,\,\zeta_{60}:16,\,\zeta_{12}:4,\, 1].$
\end{enumerate}
\end{lemma}

\begin{proof}
Note that
$[g^{30*}]=[1,\,-1.16,\,-1.4,\, 1],$ and
$$[g^{10*}]=[1,\,(\zeta_{6}:2).8,\,(\zeta_{6}:2).2,\, 1],\quad e(g^{10})=14.$$
Thus, we can apply Lemma \ref{nsym2} to $h=g^{30}$. Since $${\rm Fix}(g^{d})\subset {\rm Fix}(g^{30})$$ for any $d$ dividing 30, we see that ${\rm Fix}(g^{10})$ consists of 14 points if ${\rm Fix}(g^{30})$ is irreducible. If ${\rm Fix}(g^{30})$ is a curve $C_{9}$ of genus $9$, then $g^{10}$ acts on $C_9$ with 14 fixed points,
too many for an order 3 automorphism.  Thus ${\rm Fix}(g^{30})$ consists of a section $R$ of a
$g^{30}$-invariant elliptic fibration $$\psi:X\to {\bf P}^1$$ and
a curve $C_{10}$ of genus $10$ which is a $3$-section.  We know that 
$$X/\langle g^{30} \rangle\cong {\bf F}_4$$ a rational ruled surface. Every automorphism of ${\bf F}_4$, hence the one induced by $g$, preserves
the unique ruling, so $g$ preserves the elliptic fibration. Let $a$ and $b$ be the number of singular fibres of type $I_1$ and $II$ respectively. Then 
$$a+2b=e(X)=24, \quad 12\le a+b\le 24.$$ Note that
$g^{30}$ acts trivially on the base ${\bf P}^1$. Neither $g^5$ nor $g^6$ acts trivially on ${\bf P}^1$. Otherwise, ${\rm Fix}(g^{5})$ or ${\rm Fix}(g^{6})$ must contain the section $R$, the nodes of the nodal fibres and the cusps of the cuspidal fibres, too many, as we compute $e(g^{5})=3\pm 1$ and $e(g^{6})=4$. Our automorphism $g$ acts on the set of the base points of the $a+b$ singular fibres. An orbit of this action has length 1, 2, 3, 5, 6, 10 or 15, i.e. a divisor of 30. If an orbit has length 3, 5 or 6, then $g^5$ or $g^6$ fixes all points in the orbit, hence acts trivially on the base ${\bf P}^1$. Thus no orbit has length 3, 5, 6. If an orbit has length 15, then $a\ge 16$ and $g^2$ fixes more than two points on the base ${\bf P}^1$. We have proved that every orbit has length 1, 2, or 10. Then $g^{10}$ fixes all base points of the singular fibres. Thus it acts trivially on the base ${\bf P}^1$ and ${\rm Fix}(g^{10})$ contains $R$ and the nodes  and the cusps of the singular fibres. Since $ e(g^{10})=14$, we infer that $a=0$ and $b=12$. Then the action of $g$ on the 12 base points of the cuspidal fibres has an orbit of length 10; otherwise $g^2$ would act trivially on the base. If the remaining two points, say $\infty$ and 0, are interchanged by $g$, then $g$ fixes 2 points on the base ${\bf P}^1$ away from the 12 points, then $g^2$ fixes 4 points on the base, so acts trivially on the base. Thus $g$ fixes $\infty$ and 0.
This proves (1), (2) and (3). 

The statement (4) follows from (3) and the fact that ${\rm Fix}(g^{10})$ has Euler number 14 and is contained in $R\cup C_{10}$.

By (3) ${\rm Fix}(g)$  consists of the  
$4$ points, hence $e(g)=\Tr(g^*|H^*(X))=4$.
Again, by (3) ${\rm Fix}(g^{12})$ is a subset of $F_{\infty}\cup F_{0}$. Since $e(g^{12})=4$, ${\rm Fix}(g^{12})$ cannot contain any point other than the 4 points of ${\rm Fix}(g)$. This proves (5) and (6). 
\end{proof}

\section{Proof: the Tame Case}

Throughout this section, we assume that the characteristic $p>0$,
$p\neq 2$, 3, 5 and $g$ is an automorphism of order $60$ of a K3
surface. We first prove that $g$ cannot be purely non-symplectic.

\begin{lemma}\label{1.60} ${\rm  ord}(g)\neq 1.60$.
\end{lemma}

\begin{proof}
Suppose that ${\rm ord}(g)=1.60$. Then by Proposition
\ref{integral'} the action of $g^*$ on $H_{\rm et}^2(X,\bbQ_l)$, $l\neq{\rm char}(k)$, has
$\zeta_{60}\in \overline{\bbQ_l}$ as an eigenvalue and
$$[g^{*}]=[1,\, \zeta_{60}:16,\, \eta_1,\ldots, \eta_5]$$ where $[\eta_1,\ldots, \eta_5]$ is a combination of $\zeta_{12}:4$,  $\zeta_{10}:4$, $\zeta_{5}:4$, $\zeta_{6}:2$, $\zeta_{4}:2$, $\zeta_{3}:2$, $\pm 1$,
and the first eigenvalue corresponds to a $g$-invariant ample
divisor.

\bigskip
Claim 1: $[\eta_1,\ldots, \eta_5]\neq [\zeta_{10}:4,\,\pm 1]$,
$[\zeta_{5}:4,\,\pm 1]$.\\
Suppose that $[\eta_1,\ldots, \eta_5]=[\zeta_{10}:4,\,\pm 1]$ or
$[\zeta_{5}:4,\,\pm 1]$. Then Lefschetz fixed point formula gives
$$e(g^{30})=\Tr (g^{30*}|H^*(X))=-8$$ and ${\rm Fix}(g^{30})$ consists of $d$
smooth rational curves and a curve $C_{d+5}$ of genus $d+5$. We
have $0\le d\le 5$, since each fixed curve gives an invariant
vector in $\dim H_{\rm et}^2(X,\bbQ_l)$. Note that $e(g^2)=\Tr (g^{2*}|H^*(X))=1$.
Since ${\rm Fix}(g^2)\subset{\rm Fix}(g^{30})$, we infer that
${\rm Fix}(g^2)$ consists of a point. Note that $C_{d+5}\nsubseteq
{\rm Fix}(g^{10})$, since $e(g^{10})=16> e(g^{30})$. If $d=1, 2$ or $4$, then $g$ acts on the $d$
smooth rational curves and $g^2$ preserves at least one of them,
hence fixes at least 2 points. If $d=3$, then $g$ must rotate the
$3$ smooth rational curves and $g^{10}$ acts on the curve $C_8$
with $16$ fixed points, which is impossible. If $d=0$, then
$g^{10}$ gives an order 3 automorphism of the curve $C_5$ with
$16$ fixed points, impossible. If $d=5$, then $g$ must rotate the
$5$ smooth rational curves and $g^{5}$ preserves each of them,
hence $e(g^{5})\ge 10$. But $\Tr
(g^{5*}|H^*(X))\le 8$, contradicting the Lefschetz
fixed point formula.

\bigskip
Claim 2: $[\eta_1,\ldots, \eta_5]\neq [\zeta_{6}:2,\,\pm 1,\,\pm
1,\,\pm 1]$, $[\zeta_{3}:2,\,\pm 1,\,\pm 1,\,\pm 1]$.\\
 Suppose that
$[\eta_1,\ldots, \eta_5]=[\zeta_{6}:2,\,\pm 1,\,\pm 1,\,\pm 1]$ or
$[\zeta_{3}:2,\,\pm 1,\,\pm 1,\,\pm 1]$. This case can be handled
similarly. We see that $e(g^{30})=-8$ and ${\rm
Fix}(g^{30})$ consists of $d$ smooth rational curves and a curve
$C_{d+5}$ of genus $d+5$, $0\le d\le 5$. We also see that $e(g^2)=3$ and ${\rm Fix}(g^2)$ consists of either 3 points or
a point and a $\bbP^1$. Note that $C_{d+5}\nsubseteq {\rm
Fix}(g^{10})$, since $e(g^{10})=13> e(g^{30})$. If $d=0$ or 1, then $g^{10}$ gives an order 3
automorphism of the curve $C_{d+5}$ with at least $11$ fixed
points, which is impossible.  If $d=2$, then $g^2$ preserves 2
smooth rational curves, hence fixes at least 4 points. If $d=3$,
then $g$ must rotate the $3$ smooth rational curves and $g^{10}$
acts on the curve $C_8$ with $13$ fixed points, impossible. If
$d=4$, then $g^{3}$ preserves each of them, hence $e(g^{3})\ge 8$ or $e(g^{3})=8+e(C_9)=-8$, which is
possible only if
$[g^{*}]=[1,\,\zeta_{60}:16,\,\zeta_{3}:2,\,1,\,1,\,1]]$. Then
$e(g)=5> e(g^{2})$, but ${\rm Fix}(g)$ and
${\rm Fix}(g^{2})$ consist of isolated points and some $\bbP^1$'s.
 If $d=5$, then
$g$ must rotate the $5$ smooth rational curves and $g^{5}$
preserves each of them, hence $e(g^{5})\ge 10$. But
$\Tr (g^{5*}|H^*(X))\le 7$, 
contradicting the Lefschetz formula.

\bigskip
Claim 3: $[\eta_1,\ldots, \eta_5]\neq [(\zeta_{6}:2).2,\,\pm 1]$,
$[(\zeta_{3}:2).2,\,\pm 1]$, $[\zeta_{6}:2,\,\zeta_{3}:2,\,\pm
1]$.\\
 Suppose that $[\eta_1,\ldots, \eta_5]=[(\zeta_{6}:2).2,\,\pm
1]$, $[(\zeta_{3}:2).2,\,\pm 1]$ or
$[\zeta_{6}:2,\,\zeta_{3}:2,\,\pm 1]$. Note that $e(g^{30})=-8$ and ${\rm Fix}(g^{30})$ consists of $d$ smooth
rational curves and a curve $C_{d+5}$ of genus $d+5$, $0\le d\le
5$. We see that $e(g^2)=0$. Since ${\rm
Fix}(g^{2})\subseteq {\rm Fix}(g^{30})$, ${\rm
Fix}(g^2)=\emptyset$, thus  ${\rm Fix}(g)=\emptyset$ and
$[g^{*}]=[1,\,\zeta_{60}:16,\,(\zeta_{3}:2).2,\,-1]$. Note that
$C_{d+5}\nsubseteq {\rm Fix}(g^{10})$, since $e(g^{10})=10> e(g^{30})$. If $d=0$, then $g^{10}$
gives an order 3 automorphism of the curve $C_{5}$ with $10$ fixed
points, which is impossible. If $d=1, 2$ or 4, then $g^2$
preserves at least one smooth rational curve, hence fixes at least
2 points. If $d=3$, then $g$ must rotate the $3$ smooth rational
curves, hence $g^{15}$ acts freely on the curve $C_8$, since
$e(g^{15})=6$. But no genus 8 curve admits a free
involution. If $d=5$, then $g$ must rotate the $5$ smooth rational
curves and $g^{5}$ preserves each of them, hence $e(g^{5})\ge 10$. But $\Tr (g^{5*}|H^*(X))=0$.

\bigskip
Claim 4: $[\eta_1,\ldots, \eta_5]\neq
[\zeta_{4}:2,\,\zeta_{6}:2,\,\pm 1]$,
$[\zeta_{4}:2,\,\zeta_{3}:2,\,\pm 1]$.\\
 Suppose that
$[\eta_1,\ldots, \eta_5]=[\zeta_{4}:2,\,\zeta_{6}:2,\,\pm 1]$  or
$[\zeta_{4}:2,\,\zeta_{3}:2,\,\pm 1]$. In this case, $e(g^{30})=-12$ and ${\rm Fix}(g^{30})$ consists of $d$ smooth
rational curves and a curve $C_{d+7}$ of genus $d+7$, $0\le d\le
3$. We compute $$e(g^2)=\Tr (g^{2*}|H^*(X))=-1> e(g^{30}),$$ hence
$C_{d+7}\nsubseteq {\rm Fix}(g^{2})$. But then $e(g^2)\ge 0$.

\bigskip
Claim 5: $[\eta_1,\ldots, \eta_5]\neq [(\zeta_{4}:2).2,\,\pm 1]$.\\
Suppose that $[\eta_1,\ldots, \eta_5]=[(\zeta_{4}:2).2,\,\pm 1]$.
In this case, $e(g^{30})=-16$ and ${\rm Fix}(g^{30})$
consists of $d$ smooth rational curves and a curve $C_{d+9}$ of
genus $d+9$, $0\le d\le 1$. Since $e(g^2)=-2>
e(g^{30}),$ $C_{d+9}\nsubseteq {\rm Fix}(g^{2})$, but then $e(g^2)\ge 0$.

\bigskip
Claim 6: $[\eta_1,\ldots, \eta_5]\neq [\zeta_{4}:2,\,\pm 1,\,\pm
1,\,\pm 1]$.\\
Suppose that $[\eta_1,\ldots,
\eta_5]=[\zeta_{4}:2,\,\pm 1,\,\pm 1,\,\pm 1]$. In this case,
$e(g^{30})=-12$ and ${\rm Fix}(g^{30})$ consists of $d$
smooth rational curves and a curve $C_{d+7}$ of genus $d+7$, $0\le
d\le 3$. We compute
$$e(g^2)=\Tr (g^{2*}|H^*(X))=2> e(g^{30}),$$ hence $C_{d+7}\nsubseteq {\rm Fix}(g^{2})$ and ${\rm
Fix}(g^{2})$ consists of either 2 points or a $\bbP^1$, since
${\rm Fix}(g^{2})\subset{\rm Fix}(g^{30})$. Since ${\rm
Fix}(g)\subset{\rm Fix}(g^{2})$, we infer that $$e(g)=2\,\,\,{\rm or}\,\,\, 0.$$ By computing $[g^{15*}]$ and
$[g^{10*}]$, we see that
$$e(g)=e(g^{15})\,\,\,{\rm and}\,\,\, e(g^{10})=12.$$ If $d=0$, then $g^{10}$
gives an order 3 automorphism of the curve $C_{7}$ with $12$ fixed
points, impossible. If $d=2$, then $g^2$ preserves both
smooth rational curves, hence $e(g^2)\ge 4$. If $d=3$,
then $g^2$ cannot preserve two of the three smooth rational curves,
hence $g$ must rotate the three, then
$g^{15}$ preserves each of the three, hence
$e(g^{15})\ge 6$. If $d=1$, then $g^{15}$ acts freely
on the curve $C_8$. But no genus 8 curve admits a free involution.

\bigskip
Claim 7: $[\eta_1,\ldots, \eta_5]\neq [\pm 1,\,\pm 1,\,\pm 1,\,\pm
1,\,\pm 1]$.\\
Suppose that $[\eta_1,\ldots, \eta_5]=[\pm 1,\,\pm
1,\,\pm 1,\,\pm 1,\,\pm 1]$. In this case,\\ $e(g^{30})=-8$ and ${\rm Fix}(g^{30})$ consists of $d$ smooth
rational curves and a curve $C_{d+5}$ of genus $d+5$, $0\le d\le
5$. We also compute
$$e(g^2)=6,\,\,\,e(g^{15})=e(g), \,\,\,e(g^{10})=16.$$
Since $e(g^2)>e(g^{30})$, we see that
$C_{d+5}\nsubseteq {\rm Fix}(g^{2})$ and $$e(g^{15})=e(g)\le e(g^2)=6.$$ If $d\le
2$, then $g^{10}$ gives an order 3 automorphism of the curve
$C_{d+5}$ with $16-2d$ fixed points, which is impossible. Assume
$d\ge 4$. If $g^{15}$ preserves at least 4 of the $d$ smooth
rational curves, then $e(g^{15})\ge 8>6$. If $g^{15}$
preserves at most 2 of the $d$ smooth rational curves, then
$g^{2}$ preserves at least 4, hence $e(g^{2})\ge 8>6$.
If $g^{15}$ preserves exactly 3 of the $d$ smooth rational curves,
then $d=5$ and $g^{15}$ acts freely as an involution on the curve
$C_{10}$, a contradiction. Assume $d=3$. If $g$ rotates the $3$
smooth rational curves or fixes each of them, then $g^{15}$ fixes
each of them, hence acts freely on the curve $C_8$, a
contradiction. If $g$ fixes exactly one of the $3$ smooth rational
curves, then $g^2$ fixes each of them, hence acts freely on the
curve $C_8$, then $g$ acts freely on the curve $C_8$ and $e(g)=2$, then $g^{15}$ has $e(g^{15})=2$, hence
acts freely on the curve $C_8$. This proves the claim.

\bigskip
We may assume that $[g^{*}]=[1,\,\zeta_{60}:16,\,\zeta_{12}:4,\,\pm 1].$  Then by Lemma \ref{60}
 $$[g^{*}]=[1,\,\zeta_{60}:16,\,\zeta_{12}:4,\, 1].$$
Consider the order 5 automorphism $g^{12}$. It is  non-symplectic and the quotient
$$X':=X/\langle g^{12}\rangle$$ is a singular rational surface with
$K_{X'}$ numerically trivial. Furthermore, by Proposition \ref{diminv} Picard number $\rho(X')=6$ .

\bigskip
Claim 8: $X'=X/\langle g^{12} \rangle$ has four singular points,
one of type $\frac{1}{5}(3,3)$ and three of type
$\frac{1}{5}(2,4)$.\\
 To prove the claim, note first that ${\rm Fix}(g^{12})$ consists of
the 4 points from Lemma \ref{60}, 2 points of $R$ and 2 points of $C_{10}$. Since
$$g^{12*}\omega_X=\zeta_5\omega_X\,\,\, {\rm for\,\, some}\,\, \zeta_5\in k,$$  there are two types of local action of $g^{12}$ at a fixed
point, $\frac{1}{5}(3,3)$ and $\frac{1}{5}(2,4)$. Let $a$ and $b$
be the number of points respectively of the two types. Then
$$a+b=4.$$ Let $\varepsilon: Y\to X'$ be a minimal resolution. Then
$$K_Y=\varepsilon^*K_{X'}-\sum D_p$$
where $D_p$ is an effective $\bbQ$-divisor supported on the
exceptional set of the singular point $p\in X'$. Here $``="$ means numerical equivalence. Thus
$$K_Y^2=\sum D_p^2=-\sum K_YD_p.$$
See, e.g., Lemma 3.6 \cite{HK1} for the formulas of $D_p$ and
$K_YD_p$, which are valid not only in the complex case but also
for tame quotient singular points in positive characteristic. We
compute
$$K_Y^2= 10-\rho(Y)= 10-\{\rho(X')+a+2b\}=4-a-2b.$$ On the other hand,
$K_YD_p=\frac{9}{5}$ if $p$ is of type $\frac{1}{5}(3,3)$, and
$K_YD_p=\frac{2}{5}$ if $p$ is of type $\frac{1}{5}(2,4)$, thus
$$K_Y^2= -\frac{9}{5}a-\frac{2}{5}b.$$ Solving the system, we get $a=1$ and $b=3$. This proves the claim.

\bigskip
Now by Claim 8, we compute that
$$K_Y=-\frac{3A}{5}-\sum_{i=1}^3 \frac{A_{1i}+2A_{2i}}{5}$$
where $A$ and $A_{ji}$ are exceptional curves with $A^2=-5$,
$A_{1i}^2=-2$, $A_{2i}^2=-3$, $A_{1i}.A_{2i}=1$. If the 2 points
of $R$ are of type $\frac{1}{5}(2,4)$, then the proper transform
$R'$ of the image of $R$ in $X'$ has intersection number with
$K_Y$,
$$K_Y.R'=-\frac{1}{5}-\frac{1}{5}, \,\,-\frac{1}{5}-\frac{2}{5}\,\,{\rm or}\,\,-\frac{2}{5}-\frac{2}{5},$$
none is an integer. If the 2 points of $C_{10}$ are of type
$\frac{1}{5}(2,4)$, then the proper transform $C_{10}'$ of the
image of $C_{10}$ in $X'$ has intersection number with $K_Y$ which
cannot be an integer, a contradiction.
\end{proof}

\begin{lemma}\label{2.30} ${\rm  ord}(g)\neq 2.30$, $3.20$, $4.15$, $6.10$.
\end{lemma}

\begin{proof} These cases are much simpler than the previous one, and are contained in \cite{K}, Lemma 4.5 and 4.7.
\end{proof}

\begin{lemma}\label{5.12} If ${\rm  ord}(g)=5.12$, then
$$[g^{*}]=[1,\,\zeta_{12}:4,\,1,\,\zeta_{60}:16].$$ 
\end{lemma}

\begin{proof}
Since $g^{12}$ is symplectic of
order 5, $$[g^{12*}]=[1,\,1.5,\, (\zeta_5:4).4]$$ and for any positive integer $a$ dividing
12, ${\rm Fix}(g^{a})\subset {\rm Fix}(g^{12})$ and
$$0\le e(g^{a})\le e(g^{12})=4.$$ By
Proposition \ref{integral'}, $\zeta_{12}\in [g^{*}]$. Thus we infer
that
$$[g^{*}]=[1,\,\zeta_{12}:4,\,\pm 1,\, \eta_1,\ldots, \eta_{16}]$$ where $[\eta_1,\ldots, \eta_{16}]$ is a combination of $\zeta_{5}:4$, $\zeta_{10}:4$, $\zeta_{15}:8$,
$\zeta_{20}:8$,  $\zeta_{30}:8$, $\zeta_{60}:16$ and the first
eigenvalue corresponds to a $g$-invariant ample divisor.

\medskip
Assume that $[\eta_1,\ldots, \eta_{16}]$ contains $[\zeta_{15}:8]$
or $[\zeta_{30}:8]$. Then
$$[g^{2*}]=[1,\, (\zeta_6:2).2,\,1,\,\zeta_{15}:8,\,\tau_1,\ldots, \tau_8]$$
where $[\tau_1,\ldots, \tau_8]$ is a combination of $\zeta_{5}:4$,
$\zeta_{10}:4$, $\zeta_{15}:8$, hence $\sum\tau_j\ge -2$ and $e(g^2)=\Tr
(g^{2*}|H^*(X))=7+\sum\tau_j\ge 5$, contradicting $e(g^{2})\le 4.$

Assume that $[\eta_1,\ldots, \eta_{16}]$ contains
$[\zeta_{20}:8]$. In this case,
$$[g^{2*}]=[1,\, (\zeta_6:2).2, \, 1,\,(\zeta_{10}:4).2,\,\tau_1,\ldots,
\tau_8]$$ where $[\tau_1,\ldots, \tau_8]$ is a combination of $\zeta_{5}:4$,
$\zeta_{10}:4$, $\zeta_{15}:8$. Since  $\sum\tau_j\ge -2$,
$e(g^2)=\Tr (g^{2*}|H^*(X))=8+\sum\tau_j\ge 6$, contradicting $e(g^{2})\le 4.$ 

Assume that $[\eta_1,\ldots, \eta_{16}]$ is a combination of
$\zeta_{5}:4$, $\zeta_{10}:4$. Then
 $$[g^{6*}]=[1,\, -1.4, \, 1,\,(\zeta_{5}:4).4]$$
and $e(g^6)=\Tr (g^{6*}|H^*(X))=-4,$
contradicting $e(g^{6})\ge 0$.\\ Therefore
$[\eta_1,\ldots, \eta_{16}]=[\zeta_{60}:16]$.
Now Lemma \ref{60} applies.
\end{proof}

\bigskip
\noindent {\bf Proof of Theorem \ref{main}.} 

(1) follows from Lemmas \ref{1.60} and \ref{2.30}.

(2) We know ${\rm  ord}(g)=5.12$. By Lemma \ref{5.12} we can apply Lemma \ref{60}, and will use the elliptic structure and the notation there.  Let
$$y^2+x^3+A(t_0,t_1)x+B(t_0,t_1) = 0$$ be the Weierstrass equation of  the $g$-invariant elliptic
pencil, where $A$ (resp. $B$) is a binary form of degree  $8$
(resp. $12$). By Lemma \ref{60}, $g$ leaves invariant the
section $R$ and the action of $g$ on the base of the fibration
$\psi:X\to {\bf P}^1$ is of order 10. After a linear change of the
coordinates $(t_0,t_1)$ we may assume that $g$ acts on the base by
$$g:(t_0,t_1)\mapsto (t_0,\zeta_{60}^6t_1).$$
We know that $g$ preserves two cuspidal fibres $F_0$, $F_{\infty}$
and makes the remaining 10 cuspidal fibres to form one orbit. Thus
the discriminant polynomial $$\Delta =
-4A^3-27B^2=ct_0^2t_1^2(t_1^{10}-t_0^{10})^2$$ for some constant
$c\in k$, as it must have two double roots (corresponding to the
fibres $F_0$, $F_{\infty}$) and one orbit of double roots. We know
that the zeros of $A$ correspond to either cuspidal fibres or
nonsingular fibres with ``complex multiplication'' automorphism of
order 6. Since this set is invariant with respect to the order 10
action of $g$ on the base, we see that the only possibility is $A
= 0$. Then the above Weierstrass equation can be written in
the form
$$y^2+x^3+at_0t_1(t_1^{10}-t_0^{10}) = 0$$ 
for some constant $a$. A suitable
linear change of variables makes $a=1$ without changing the action
of $g$ on the base. Thus $$X \cong X_{60}$$ as an elliptic
surface. We may assume that
$$g^*\Big(\frac{dx\wedge dt}{y}\Big)=\zeta_{60}^5\frac{dx\wedge dt}{y}$$ for some primitive 12th root of unity $\zeta_{60}^5$. Here a choice of such a root of unity is equivalent to a choice of a generator of the cyclic group $\langle g\rangle$. Since $g^{10}$ is of order 6 and acts trivially
on the base, it is a complex multiplication of order 6 on a general fibre, so
$$g^{10}(x, y, t_0, t_1)=(\zeta_6^2x, \zeta_6^3y, t_0, t_1).$$ Note that $${\rm Fix}(g)=\{{\rm the\,
two\,cusps\, of}\,F_0\,{\rm and}\, F_{\infty}\}\cup (R\cap
F_0)\cup (R\cap F_{\infty}).$$ Analysing the local action of $g$
at the fixed point $(x, y, t_0, t_1)=(0,0,1,0)$, the cusp of
$F_0$, we infer that
$$g(x, y, t_0, t_1)=(\zeta_{60}^2x, \zeta_{60}^3y, t_0,
\zeta_{60}^6t_1).$$ Here we first determine the linear terms, then
see that the higher degree terms must vanish. This completes the
proof of Theorem \ref{main} in the positive characteristic case.

\section{Proof: the Complex Case}

Throughout this section, $X$ is a complex K3 surface.

A non-projective K3 surface cannot admit a non-symplectic
automorphism of finite order (see \cite{Ueno}, \cite{Nik}), and
its automorphisms of finite order are symplectic, hence of order
$\le 8$. Thus we may assume that $X$ is projective. The proofs of
Lemma \ref{nsym2}, \ref{60}, \ref{1.60}, \ref{2.30} and \ref{5.12} go word for word, once
the $l$-adic cohomology $H^2_{\rm et}(X,\bbQ_l)$ is replaced by the
integral singular cohomology $H^2(X,\bbZ)$,
and Proposition \ref{trace} by the usual topological Lefschetz
fixed point formula. ``Proof of Theorem \ref{main}" also
goes word for word.




\end{document}